\documentclass[11pt,a4paper]{article}
\usepackage[letterpaper,top=2.5cm,bottom=2.5cm,left=2.5cm,right=2.5cm,marginparwidth=1.75cm]{geometry}
\DeclareMathAlphabet{\mathpzc}{OT1}{pzc}{m}{it}
\usepackage{amsmath,amsfonts,amsthm,amssymb}
\usepackage{amsmath,amssymb,color}
\usepackage{graphicx,fancyhdr}
\usepackage{mathrsfs,float}
\usepackage{multirow}
\usepackage{dsfont}
\usepackage{subfigure}
\usepackage{booktabs}
\usepackage{yfonts}
\usepackage{natbib}
\usepackage{authblk}
\usepackage{caption}
\usepackage{geometry}
\usepackage{bbm}
\usepackage{bm}
\makeatletter
\usepackage{hyperref}

\numberwithin{equation}{section}
\newtheorem{thm}{Theorem}[section]
\newtheorem{rem}[thm]{Remark}

\newtheorem{prop}[thm]{Proposition}
\newtheorem{prob}[thm]{Problem}
\newtheorem{defi}[thm]{Definition}

\newtheorem{lem}[thm]{Lemma}

\def\la{\lambda}

\def\om{\omega}
\def\f{\frac}
\def\nn{\nonumber}

\def\tri{\triangle}
\def\beqnx{\begin{eqnarray*}} \def\eqnx{\end{eqnarray*}}

\theoremstyle{definition}

\title{Uniqueness for an inverse coefficient problem of a weakly coupled parabolic system}
\author[a]{Caixuan Ren}
\author[b]{Kai Yu\thanks{Corresponding author: yukaimailbox@163.com}}
\author[b]{Zhiyuan Li}
\affil[a]{School of Mathematics and Statistics, Donghua University, Shanghai 201620, China}
\affil[b]{School of Mathematics and Statistics, Ningbo University, Ningbo 315211, China}

\date{}

\begin{document}
\maketitle

\renewcommand{\thefootnote}{\fnsymbol{footnote}}
\renewcommand{\thefootnote}{\arabic{footnote}}

\begin{abstract}

This paper considers the weakly coupled  parabolic system  $\partial_t u-\partial^2_xu +P(x)u=0$ with the homogeneous Neumann boundary condition, where \(P(x)\) is a \(2\times2\) symmetric real-valued function matrix. Under the assumption that the initial value \(a(x)\) is a generating element (i.e., it has a nonzero inner product with every eigenfunction), we prove that the coefficient matrix  $ P(x)$ is uniquely determined by the boundary observation
$u(0, t)$, $u(1, t)$, $0 < t < T$. The proof relies on the eigenfunction expansion of the solution to the initial-boundary value problem and an extension of the Gel'fand-Levitan theory to the parabolic system. \medskip

 \medskip


\end{abstract}

\section{Introduction}

Let \(\Sigma:=(0,1)\times(0,\infty)\) and consider the following weakly coupled parabolic system subject to the Neumann boundary condition:
\begin{equation}\label{equ}
\left\{\!\begin{alignedat}{5}
& \partial_t  u(x,t)-\partial_x^{2}  u(x,t) + P(x)u(x,t)=0, \quad &&(x,t)\in \Sigma,\\
&\partial_x  u(0,t)=\partial_x  u(1,t)= 0, \quad &&t \in (0,\infty), \\
& u(x,0)= a(x),  \quad && x\in(0,1),
\end{alignedat}\right.
\end{equation}
where $ u=(u_1,u_2)^{\rm T}$, $P(x)=(p_{ij}(x))$ is a symmetric real-valued function matrix with $p_{ij}(x)\in C^1([0,1]) $, and $a(x)=\left(a_1(x),a_2(x)\right)^{\rm T}$ with $a_i(x) \in  L^2(0,1)$ for ${1\leq i,j\leq 2}$.

The weakly coupled parabolic system \eqref{equ} is widely used to model various practical applications, for example: population dynamics of interacting species and insect dispersal \cite{Murray2002}; the emergence and growth of cancer and tumor-induced angiogenesis \cite{Chakrabarty2005}; and two-compound chemical reactions in the presence of diffusion that produce spatial patterns \cite{Turing1952}. In the forward problem, given the initial value \(a(x)\) and the coefficient matrix \(P(x)\), the existence and uniqueness of the solution are classical results \cite{Friedman2008, Henry1981}. This paper is concerned with the corresponding inverse problem:

\begin{prob} Given the initial value $ a(x)$ and under certain assumptions, can we uniquely determine $P(x)$ from the data $\{u(0,t), u(1,t)\}$ for $0<t< T$ with some finite $0<T<\infty$?
\end{prob}

It is obvious that the answer is negative without any further assumption on the pair $\{P, a\}$. Since a trivial initial value \(a\equiv 0\) leads to the identically zero solution regardless of the choice of \(P\), a non-trivial condition must be imposed on the initial value. Hence, we introduce the following definition.

\begin{defi}
Let $L^2(0,1)$ be the Hilbert space of square-integrable real-valued functions 
on the interval $(0,1)$, and let $\{L^2(0,1)\}^2$ be the corresponding product 
space. For $u = (u_1, u_2)^{\rm T}\in \{L^2(0,1)\}^2$, 
we define the inner product and the associated norm by
\[
(u,v)_{\{L^2(0,1)\}^2} = \int_0^1 \bigl( u_1(x)v_1(x) + u_2(x)v_2(x) \bigr)\,dx, \quad
\|u\|^2_{\{L^2(0,1)\}^2} = {(u,u)_{\{L^2(0,1)\}^2}}.
\]
For brevity, we write $(\cdot,\cdot)$ and $\|\cdot\|$ when the space is clear from the context.

Denote by $\mathcal{A}_P$ the operator 
$-\frac{\partial^2}{\partial x^2}+P(x)$ in $\{L^2(0,1)\}^2$ subject to homogeneous 
Neumann boundary conditions at $x=0$ and $x=1$. The corresponding eigenvalues 
and eigenfunctions are denoted by $\{\lambda_n\}_{n=1}^{\infty}$ and 
$\{\psi_n(x)\}_{n=1}^{\infty}$, respectively. Assume that each eigenvalue 
$\lambda_n$ is simple, and $\lambda_1<\lambda_2<\cdots<\lambda_n<\cdots$. 
An initial value $a\in\{L^2(0,1)\}^2$ is said to be a \textit{generating element} 
with respect to $\mathcal{A}_P$ if $(a,\psi_n(\cdot))\neq 0$ for all 
$n\in \mathbb{N}:=\{1,2,\cdots\}$.
\end{defi}

\begin{rem}
A sufficient condition for the simplicity of the eigenvalues of $\mathcal{A}_P$ 
is that the matrix $P(x)$ can be diagonalized. More precisely, if there exists an invertible constant matrix $U$ 
such that 
\[
U^{-1} P(x) U = \operatorname{diag}(\tilde{p}_1(x), \tilde{p}_2(x)),
\]
and the two scalar Sturm--Liouville problems
\[
y_i''(x) - \tilde{p}_i(x) y_i(x) = -\lambda y_i(x), \quad 
y_i'(0) = y_i'(1) = 0, \qquad i = 1,2,
\]
have disjoint spectra, then each eigenvalue of $\mathcal{A}_P$ is simple.
\end{rem}

For the purpose of the inverse problem,  we introduce a system \eqref{eqv} associated with a different coefficient matrix $Q(x)$:
\begin{equation}\label{eqv}
\left\{\!\begin{alignedat}{3}
& \partial_t v(x,t)-\partial_x^2 v(x,t)+Q(x)v(x,t)=0 \quad && (x,t)\in \Sigma,\\
& \partial_x v(0,t) =\partial_x v(1,t)=0 \quad &&t \in (0,\infty), \\
& v(x,0)=a(x) \quad && x\in(0,1),
\end{alignedat}\right.
\end{equation}
where $Q(x)=(q_{ij}(x))_{1\leq i,j \leq 2}\in \{C^1([0,1])\}^{2\times2}$ and $v=(v_1,v_2)^{\rm T}$.

Analogous to $\mathcal{A}_P$, let $\mathcal{A}_Q$ denote the operator $-\frac{\partial^2}{\partial x^2}+Q(x)$ in $\{L^2(0,1)\}^2$ 
with homogeneous Neumann boundary conditions. Its eigenvalues and eigenfunctions 
are denoted by $\{\mu_m\}_{m=1}^{\infty}$ and $\{\phi_m(x)\}_{m=1}^{\infty}$, 
respectively. We assume that each eigenvalue $\mu_m$ is simple and ordered such 
that $ \mu_1 < \mu_2 < \cdots<\mu_m<\cdots$. For each $m,n\in\mathbb{N}$, the eigenfunctions of both
$\mathcal{A}_P$ and $\mathcal{A}_Q$ are normalized at the left endpoint $x=0$ such that 
\[
\|\psi_n(0)\|_2 = \|(b^1_n,b^2_n)^{\rm T}\|_2 = 1, \quad 
\|\phi_m(0)\|_2 = \|(d^1_m,d^2_m)^{\rm T}\|_2 = 1 .
\]
Here $\|\cdot\|_2$ denotes the Euclidean norm in $\mathbb{R}^2$.

In this work, we take a approach by extending the classical 
Gel'fand-Levitan theory to the weakly coupled parabolic 
system. The proof combines an eigenfunction expansion of the 
solution with a transformation formula linking the spectral data of 
$\mathcal{A}_P$ and $\mathcal{A}_Q$ via hyperbolic equations. Under the generating element assumption, this yields the 
following uniqueness theorem.

\begin{thm}\label{th1}
Consider the systems \eqref{equ} and \eqref{eqv} for $P(x), Q(x)\in \{C^1([0,1])\}^{2\times2}$. Let $a\in \{L^2(0,1)\}^2$ and assume that $a(x)$ is a generating element with respect to both $\mathcal{A}_P$ and $\mathcal{A}_Q$. Then, the equality of the boundary data
\begin{equation}\label{data1}
u(0,t)=v(0,t), \quad u(1,t)=v(1,t), \quad 0<t< T
\end{equation}
implies that $P(x)=Q(x)$ for $x\in [0,1]$.
\end{thm}

In the scalar case (i.e., when \(u\) has a single component), the inverse 
coefficient problem for parabolic equations has been studied extensively. 
Early investigations by Cannon \cite{Cannon1964} focused on determining space-dependent thermal conductivity 
in heat equations. Klibanov \cite{Klibanov1992} established uniqueness results 
using Carleman estimates to further developments. Parallel to these efforts, the Gel'fand--Levitan theory, originally developed for Sturm--Liouville operators 
\cite{Gelfand1951, Levitan1987, Marchenko1977}, was adapted to parabolic inverse problems: Pierce \cite{Pierce1979} identified eigenvalues and 
coefficients uniquely from one point over a finite time interval. Suzuki and Murayama \cite{Suzuki1980,Suzuki1983} established a unique identification theorem for both coefficients and initial values based on the boundary data within a finite time interval, while Murayama 
\cite{Murayama1981} provided a systematic treatment using the Gel'fand--Levitan framework. A comprehensive summary of uniqueness, stability, and numerical 
methods for parabolic inverse problems is given in the monograph by Isakov 
\cite{Isakov2006}, and the application of Carleman estimates to coefficient 
inverse problems is reviewed by Yamamoto \cite{Yamamoto2009}.

In contrast, research on inverse problems for coupled parabolic systems remains 
relatively limited. These contributions include Akhundov 
\cite{Akhundov2006} established uniqueness results for certain strongly 
coupled parabolic systems. Benabdallah et al. \cite{Benabdallah2009} subsequently 
considered a two-component parabolic system and, using Carleman estimates, 
obtained stability when only a single component is measured. In the one-dimensional 
setting, Cristofol et al. \cite{Cristofol2013} treated the case of discontinuous 
conductivities. Wu and Yu \cite{Wu2017} later proved H\"older stability for a strongly coupled reaction-diffusion system. More recently, Allal et al. \cite{Allal2020} 
established Lipschitz stability for degenerate coupled parabolic systems using 
locally distributed observations of a single component. Li and Sun \cite{Li2026} 
investigated simultaneous uniqueness and developed numerical inversion schemes 
for a coupled diffusion system. 
The scarcity of results in this direction stems largely from the difficulty of 
decoupling the observations: in a coupled system, measurements of one component 
inevitably entangle information from all unknown coefficients, rendering 
traditional spectral or scalar techniques much less effective.

The main contribution of this paper is the extension of the Gel'fand-Levitan 
theory to weakly coupled parabolic systems. This extension builds upon the 
spectral theory of the vector Sturm-Liouville operator, in particular on spectral asymptotics, the completeness of eigenfunction 
expansions, and the inverse spectral problem for linear differential operators (see \cite{Amrein2005, Carlson1993, Yurko2000, Yurko2006}). 
The proof of uniqueness rests on the transformation formula \eqref{tranfor}, 
which connects the eigenfunctions of $\mathcal{A}_P$ and $\mathcal{A}_Q$ through 
an integral kernel $K$ satisfying hyperbolic equations 
(see \cite{Bitsadze2014, Courant1962, Polyanin2008}). Expanding $u$ and $v$ in 
eigenfunction series, the condition \eqref{data1} implies 
$\lambda_n = \mu_n$, $\phi_n(0) = c_n\psi_n(0)$ and $\phi_n(1) = c_n\psi_n(1)$ for all $n$. These 
relations translate into the boundary conditions $K(1,y) = 0$ and 
$K_x(1,y) = 0$ on the kernel. Uniqueness results for the hyperbolic equation 
then force $K \equiv 0$, and the identity 
$K(x,x) = \frac{1}{2}\int_0^x (Q(s)-P(s))\,ds$ yields $P(x) = Q(x)$.

Our paper is organized as follows. Section \ref{sec2} presents several propositions required to establish the Gel'fand-Levitan theory for the weakly coupled parabolic system. Section \ref{sec3} contains the proof of the main uniqueness result based on the eigenfunction expansion of the solution and the Gel'fand-Levitan theory. Section \ref{sec4} offers concluding remarks. The appendix is devoted to the proof of the propositions introduced in Section \ref{sec2}.

\section{Preliminary for the Gel'fand--Levitan theory}
\label{sec2}

In this section we collect several results on a hyperbolic 
system, which will serve as essential propositions for the construction of the 
transformation kernel in the Gel'fand--Levitan framework developed in 
Section \ref{sec3}.

Let $\Omega\subset\mathbb{R}^2$ be the interior of an isosceles right triangle 
$\triangle ABC$ with $\overline{AC}=\overline{BC}$ and $\angle ACB=\frac{\pi}{2}$, 
and assume that $AB$ is parallel to one of the coordinate axes. 
For a vector-valued function $K_v=(K_{v_1},K_{v_2},K_{v_3},K_{v_4})^{\rm T}\in 
\{C^2(\bar{\Omega})\}^{4\times1}$, we denote by $\|K_v\|_{C^k(\bar{\Omega})}$ 
($k=0,1,2$) the maximum of the absolute values of all its components and of 
their partial derivatives up to order $k$ over $\bar{\Omega}$ \cite{AF2003}. An analogous 
convention is adopted for a matrix-valued function 
$r(x,y)=(r_{ij}(x,y))_{4\times4}\in \{C^1(\bar{\Omega})\}^{4\times4}$.

We consider the hyperbolic system
\begin{equation}\label{K}
\frac{\partial^2 K_v(x,y)}{\partial x^2} - \frac{\partial^2 K_v(x,y)}{\partial y^2} 
= r(x,y)K_v(x,y), \quad (x,y)\in\Omega.
\end{equation}
The following three propositions concern the unique solvability of 
\eqref{K} under various boundary conditions prescribed on the characteristic 
segments $AC$, $BC$, and $AB$. Their proofs are constructive and rely on 
successive approximations; detailed arguments are provided in the Appendix.

\begin{prop}\label{goursat1}
For any $f\in\{C^2(\overline{AC})\}^{4\times1}$ and $g\in\{C^2(\overline{BC})\}^{4\times1}$ 
with the compatibility condition $f|_C=g|_C$, there exists a unique solution $K_v\in\{C^2(\bar{\Omega})\}^{4\times1}$ 
of \eqref{K} satisfying
\begin{equation}\label{K41bc}
K_v|_{AC}=f,\qquad K_v|_{BC}=g.
\end{equation}
Moreover, there exists a constant $M>0$ depending only on $\|r\|_{C^1(\bar{\Omega})}$ 
such that
\begin{equation}\label{estKc1}
\|K_v\|_{C^2(\bar{\Omega})} \leq M\bigl(\|f\|_{C^2(\overline{AC})} + \|g\|_{C^2(\overline{BC})}\bigr).
\end{equation}
\end{prop}

\begin{prop}\label{goursat2}
For any $f\in\{C^2(\overline{AB})\}^{4\times1}$ and $g\in\{C^1(\overline{AB})\}^{4\times1}$, 
there exists a unique solution $K_v\in\{C^2(\bar{\Omega})\}^{4\times1}$ of \eqref{K} satisfying
\begin{equation}\label{K42bc}
K_v|_{AB}=f,\qquad \frac{\partial K_v}{\partial\nu}\Big|_{AB}=g,
\end{equation}
where $\nu$ denotes the outward unit normal to $AB$. Furthermore, there exists 
a constant $M>0$ depending only on $\|r\|_{C^1(\bar{\Omega})}$ such that
\begin{equation}\label{estKc2}
\|K_v\|_{C^2(\bar{\Omega})} \leq M\bigl(\|f\|_{C^2(\overline{AB})} + \|g\|_{C^1(\overline{AB})}\bigr).
\end{equation}
\end{prop}

\begin{prop}\label{goursat4}
For any $f\in\{C^2(\overline{AC})\}^{4\times1}$ and $g\in\{C^1(\overline{AB})\}^{4\times1}$, 
there exists a unique solution $K_v\in\{C^2(\bar{\Omega})\}^{4\times1}$ of \eqref{K} satisfying
\begin{equation}\label{K44bc}
K_v|_{AC}=f,\qquad \frac{\partial K_v}{\partial\nu}\Big|_{AB}=g.
\end{equation}
Moreover, there exists a constant $M>0$ depending only on $\|r\|_{C^1(\bar{\Omega})}$ 
such that
\begin{equation}\label{estKc3}
\|K_v\|_{C^2(\bar{\Omega})} \leq M\bigl(\|f\|_{C^2(\overline{AC})} + \|g\|_{C^1(\overline{AB})}\bigr).
\end{equation}
\end{prop}

\section{Uniqueness}\label{sec3}

In order to prove $P(x)=Q(x)$, we introduce the Gel'fand-Levitan theory for the vectorial Sturm-Liouville operator $\mathcal{A}_P$, $\mathcal{A}_Q$. To this end, we propose the following two lemmas. The propositions proposed in the previous section will be the key ingredient in the
proof of Lemma \ref{lemma1}, while Lemma \ref{lemma2} follows immediately from the integration by parts.

\begin{lem}\label{lemma1}
For given matrix-valued functions $P(x), Q(x) \in \{C^1([0,1])\}^{2\times 2}$, 
there exists a unique solution $K \in \{C^2(\bar{D})\}^{2\times 2}$ of the system
\begin{equation}\label{KK}
\begin{cases}
K_{xx}(x,y) - K_{yy}(x,y) + K(x,y)P(y) = Q(x)K(x,y), & (x,y) \in \bar{D}, \\[4pt]
K(x,x) = \dfrac{1}{2}\displaystyle\int_0^x \bigl(Q(s) - P(s)\bigr)\,ds, & x \in [0,1], \\[6pt]
K_y(x,0) = 0, & x \in [0,1],
\end{cases}
\end{equation}
where $D = \{(x,y) \mid 0 < x < 1,\; 0 < y < x\}$.
\end{lem}

\begin{proof}
To handle the system in a unified manner, we vectorize the matrix equation. 
Let $\operatorname{vec}: \mathbb{R}^{2\times 2} \to \mathbb{R}^{4\times1}$ be the column-wise 
vectorization operator. Define
\[
K_v(x,y) = \operatorname{vec}(K(x,y)), \qquad
Q_v(s) = \operatorname{vec}(Q(s)), \qquad
P_v(s) = \operatorname{vec}(P(s)).
\]
Then, the matrix equation  
\eqref{KK} is equivalent to the following vector function
equation
\begin{equation}\label{KKv}
\begin{cases}
\partial_x^2 K_v(x,y) - \partial_y^2 K_v(x,y) = r(x,y)\,K_v(x,y), & (x,y) \in \bar{D}, \\[4pt]
K_v(x,x) = \dfrac{1}{2}\displaystyle\int_0^x \bigl(Q_v(s) - P_v(s)\bigr)\,ds, & x \in [0,1], \\[6pt]
\partial_y K_v(x,0) = 0, & x \in [0,1],
\end{cases}
\end{equation}
where the matrix $r(x,y)$ is given explicitly by
\[
r(x,y) = \begin{pmatrix}
q_{11}(x)-p_{11}(y) & -p_{12}(y) & q_{12}(x) & 0 \\
-p_{12}(y) & q_{11}(x)-p_{22}(y) & 0 & q_{12}(x) \\
q_{12}(x) & 0 & q_{22}(x)-p_{11}(y) & -p_{12}(y) \\
0 & q_{12}(x) & -p_{12}(y) & q_{22}(x)-p_{22}(y)
\end{pmatrix}.
\]
To show the existence of $K_v$, extend $Q(x)=(q_{ij}(x)) \in \{C^1([0,1])\}^{2\times2}$ to a function $\tilde{Q}(x)=(\tilde{q}_{ij}(x)) \in \{C^1([0,2])\}^{2\times2}$. 
Define the extended triangular region
\[
\tilde{D} = \bigl\{ (x,y) \mid 0 < y < 1,\, y < x < 2-y \bigr\}.
\]
Applying Proposition \ref{goursat4} on $\tilde{D}$, we obtain a solution 
$\tilde{K}_v \in \{C^2(\bar{\tilde{D}})\}^{4\times1}$ of
\begin{equation*}
\begin{cases}
\partial_x^2 \tilde{K}_v - \partial_y^2 \tilde{K}_v = \tilde{r}(x,y)\,\tilde{K}_v, & (x,y) \in {\bar{\tilde{D}}}, \\[4pt]
\tilde{K}_v(x,x) = \dfrac{1}{2}\displaystyle\int_0^x \bigl(Q_v(s) - P_v(s)\bigr)\,ds, & x \in [0,1], \\[6pt]
\partial_y \tilde{K}_v(x,0) = 0, & x \in [0,2],
\end{cases}
\end{equation*}
where $\tilde{r}(x,y)$ is defined analogously by
\[
\tilde r(x,y) = \begin{pmatrix}
\tilde q_{11}(x)-p_{11}(y) & -p_{12}(y) & \tilde q_{12}(x) & 0 \\
-p_{12}(y) & \tilde q_{11}(x)-p_{22}(y) & 0 & \tilde q_{12}(x) \\
\tilde q_{12}(x) & 0 & \tilde q_{22}(x)-p_{11}(y) & -p_{12}(y) \\
0 & \tilde q_{12}(x) & -p_{12}(y) & \tilde q_{22}(x)-p_{22}(y)
\end{pmatrix}.
\]
Thus $K_v := \tilde{K}_v|_{\bar{D}}$ belongs to $\{C^2(\bar{D})\}^{4\times1}$ 
and satisfies \eqref{KKv}, thereby establishing existence.

For the uniqueness, suppose $K_v \in \{C^2(\bar{D})\}^{4\times1}$ solves \eqref{KKv} with 
homogeneous data (i.e., $Q_v \equiv P_v \equiv 0$ on the right-hand side of the 
second equation). We must show $K_v \equiv 0$ on $\bar{D}$.  To this end, split $D$ into
\[
\Omega_1 = \bigl\{ (x,y) \mid 0 < y < \tfrac12,\; y < x < 1-y \bigr\}, \quad
\Omega_2 = D \setminus \Omega_1.
\]
Applying Proposition \ref{goursat4} to the subdomain $\Omega_1$, we deduce $K_v \equiv 0$ 
on $\Omega_1$.  Since $K_v \in \{C^2(\bar{D})\}^{4\times1}$, the function and its first derivatives 
vanish on the common boundary $\partial\Omega_1 \cap D$.  Consequently, 
Proposition \ref{goursat1} yields 
$K_v \equiv 0$ on $\Omega_2$.  Hence $K_v \equiv 0$ on the whole domain $\bar{D}$. Therefore, this complete the proof.
\end{proof}

\begin{lem}\label{lemma2}
Let $K(x,y)$ be the kernel constructed in Lemma \ref{lemma1}. For each 
$\lambda\in\mathbb{R}$, let $\psi(x)=\psi(x;\lambda)$ be the solution of
\begin{equation*}\label{psl}
\left\{\!\begin{alignedat}{2}
& -\psi''(x)+P(x)\psi(x)=\lambda \psi (x), \\
&\psi(0)=\!\begin{alignedat}{2}
\left ( \begin{matrix}
b^1 \\
b^2
\end{matrix}\right )
\end{alignedat},
\psi^{\prime}(0)= \!\begin{alignedat}{2}
\left ( \begin{matrix}
0 \\
0
\end{matrix}\right )
\end{alignedat}
\end{alignedat}\right.
\end{equation*}
normalized by $\|\psi(0)\|_2 = \|(b^1, b^2)^{\rm T}\|_2 = 1$. Define
\begin{equation}\label{tranfor}
\Phi(x) = \psi(x) + \int_0^x K(x,y)\,\psi(y)\,dy.
\end{equation}
Then $\Phi(x)$ satisfies
\begin{equation}\label{qslc}
\left\{\!\begin{alignedat}{2}
& -\Phi''(x)+Q(x)\Phi(x)=\lambda \Phi (x), \\
&\Phi(0)=\!\begin{alignedat}{2}
\left ( \begin{matrix}
b^1 \\ b^2
\end{matrix}\right )
\end{alignedat},
\Phi^{\prime}(0)= \!\begin{alignedat}{2}
\left ( \begin{matrix}
0 \\ 0
\end{matrix}\right )
\end{alignedat}.
\end{alignedat}\right.
\end{equation}
\end{lem}

\begin{proof}
Since ${\Phi}(x) = \psi(x) + \int_0^x K(x,y)\psi(y)\,dy$, we get
\begin{equation*}\label{Phidef}
{\Phi}(0) = \psi(0) = (b^1,b^2)^{\rm T}.
\end{equation*}
Differentiating with respect to $x$ gives
\[
{\Phi}'(x) = \psi'(x) + K(x,x)\psi(x) + \int_0^x K_x(x,y)\psi(y)\,dy
\]
and evaluating at $x=0$ yields ${\Phi}'(0) = \psi'(0) = (0,0)^{\rm T}$. 

Next, we verify that ${\Phi}$ satisfies $-\Phi''(x)+Q(x)\Phi(x)=\lambda \Phi (x)$. Differentiating the integral expression yields
\begin{align*}
{\Phi}''(x) &= \psi''(x) + \frac{d}{dx}\bigl[K(x,x)\psi(x)\bigr] + K_x(x,x)\psi(x) + \int_0^x K_{xx}(x,y)\psi(y)\,dy.
\end{align*}
From Lemma \ref{lemma1}, $K$ satisfies
\begin{equation*}\label{kernel-eq}
K_{xx}(x,y) - K_{yy}(x,y) = \bigl(Q(x) - P(y)\bigr)K(x,y)
\end{equation*}
with boundary conditions $K_y(x,0) = 0$ and $K(x,x) = \frac{1}{2}\int_0^x (Q(s)-P(s))\,ds$.
Using $\psi''(y) = \lambda\psi(y) - P(y)\psi(y)$ and integrating by parts, we obtain

\begin{align*}
\int_0^x K_{xx}(x,y)\psi(y)\,dy 
&= \int_0^x K_{yy}(x,y)\psi(y)\,dy + \int_0^x \bigl(Q(x)-P(y)\bigr)K(x,y)\psi(y)\,dy \\
&= \Bigl[ K_y(x,y)\psi(y) - K(x,y)\psi'(y) \Bigr]_{y=0}^{y=x} 
+ \int_0^x K(x,y)\psi''(y)\,dy \\
&\quad + Q(x)\int_0^x K(x,y)\psi(y)\,dy - \int_0^x P(y)K(x,y)\psi(y)\,dy \\
&= K_y(x,x)\psi(x) - K(x,x)\psi'(x) + \int_0^x K(x,y)\bigl(P(y)\psi(y) - \lambda\psi(y)\bigr)\,dy \\
&\quad + Q(x)\int_0^x K(x,y)\psi(y)\,dy - \int_0^x P(y)K(x,y)\psi(y)\,dy \\
&= K_y(x,x)\psi(x) - K(x,x)\psi'(x) - \lambda\int_0^x K(x,y)\psi(y)\,dy \\
&\quad + Q(x)\int_0^x K(x,y)\psi(y)\,dy .
\end{align*}
Notice that $-\psi''(x) + Q(x)\psi(x) = \lambda\psi(x) - P(x)\psi(x) + Q(x)\psi(x)$ and $\frac{d}{dx}K(x,x) = K_x(x,x) + K_y(x,x) = \frac{1}{2}(Q(x)-P(x))$. Substitute these into $-{\Phi}''(x)+Q(x){\Phi}(x)$ yields
\begin{equation*}
\begin{aligned}
-{\Phi}''(x) + Q(x){\Phi}(x) 
&= -\psi''(x) - \frac{d}{dx}\bigl[K(x,x)\psi(x)\bigr] - K_x(x,x)\psi(x) \\
&\quad - K_y(x,x)\psi(x) + K(x,x)\psi'(x) +\lambda\int_0^x K(x,y)\psi(y)\,dy \\
&\quad - Q(x)\int_0^x K(x,y)\psi(y)\,dy + Q(x)\psi(x) + Q(x)\int_0^x K(x,y)\psi(y)\,dy\\
&= -\psi''(x) - \frac{d}{dx}K(x,x)\psi(x)-K(x,x)\psi'(x) - K_x(x,x)\psi(x) \\
&\quad - K_y(x,x)\psi(x) + K(x,x)\psi'(x) +\lambda\int_0^x K(x,y)\psi(y)\,dy  + Q(x)\psi(x) \\
&= -\psi''(x) +P(x)\psi(x)  +\lambda\int_0^x K(x,y)\psi(y)\,dy=\la\Phi(x).\\
\end{aligned}    
\end{equation*}
Therefore, the proof is completed.
\end{proof}

\begin{proof}[\bf Proof of Theorem \ref{th1}]

By eigenfunction expansion, the solutions of \eqref{equ} and \eqref{eqv} are given by
\begin{align}\label{uexpan}
u(x,t) = \sum_{n=1}^\infty e^{-\lambda_n t} \, \frac{(a,\psi_n)}{\rho_n} \, \psi_n(x), \quad
v(x,t) = \sum_{m=1}^\infty e^{-\mu_m t} \, \frac{(a,\phi_m)}{\sigma_m} \, \phi_m(x),
\end{align}
where $\rho_n = \|\psi_n\|^2 = (\psi_n,\psi_n)$ and $\sigma_m = \|\phi_m\|^2 = (\phi_m,\phi_m)$.

By the condition \eqref{data1} and the analytic continuation in $t$, we have $u(0,t)=v(0,t)$, $u(1,t)=v(1,t)$ for $0<t<\infty$. From \eqref{uexpan}, one obtains
\begin{align}\label{data1expan}
&\sum_{n=1}^\infty e^{-\lambda_nt}\f{(a,\psi_n)}{\rho_n}\psi_n(0)=\sum_{m=1}^\infty e^{-\mu_mt}\f{(a,\phi_m)}{\sigma_m}\phi_m(0), \quad 0<t<\infty, \\\label{data2expan}
&\sum_{n=1}^\infty e^{-\lambda_nt}\f{(a,\psi_n)}{\rho_n}\psi_n(1)=\sum_{m=1}^\infty e^{-\mu_mt}\f{(a,\phi_m)}{\sigma_m}\phi_m(1), \quad 0<t<\infty.
\end{align}
According to \eqref{data1expan}, we know
$$e^{-\mu_{1}t}\f{(a,\phi_{1})}{\sigma_{1}}\phi_{1}(0)=\sum_{n=1}^{\infty} e^{-\lambda_nt}\f{(a,\psi_n)}{\rho_n}\psi_n(0)-\sum_{m=2}^\infty e^{-\mu_mt}\f{(a,\phi_m)}{\sigma_m}\phi_m(0),$$
thus
\begin{equation}\label{lam1mu11}
e^{-\mu_{1}t}\f{|(a,\phi_{1})|}{\sigma_{1}}\|\phi_{1}(0)\|_2
\leq\sum_{n=1}^{\infty} e^{-\lambda_nt}\f{|(a,\psi_n)|}{\rho_n}\|\psi_n(0)\|_2+\sum_{m=2}^\infty e^{-\mu_mt}\f{|(a,\phi_m)|}{\sigma_m}\|\phi_m(0)\|_2.    
\end{equation}
In view of $ \|\psi_n(0)\|_2=\|\phi_m(0)\|_2=1$, \eqref{lam1mu11} implies
\begin{align}\label{lam1mu12}
e^{-\mu_{1}t}\f{|(a,\phi_{1})|}{\sigma_{1}}
&\leq\sum_{n=1}^{\infty} e^{-\lambda_nt}\f{|(a,\psi_n)|}{\rho_n}+\sum_{m=2}^\infty e^{-\mu_mt}\f{|(a,\phi_m)|}{\sigma_m}.
\end{align}
From H$\ddot{\rm o}$lder inequality and the relation  \eqref{lam1mu12}, it follows that
\begin{equation*}\label{lam1mu13}
\begin{aligned}
\nn e^{-\mu_{1}t}\f{|(a,\phi_{1})|}{\sigma_{1}}
&\leq e^{-\lambda_1t}\sum_{n=1}^{\infty} e^{-(\lambda_n-\lambda_1)t}\f{|(a,\psi_n)|}{\rho_n}+e^{-\mu_2t}\sum_{m=2}^\infty e^{-(\mu_m-\mu_2)t}\f{|(a,\phi_m)|}{\sigma_m}\\
&\leq e^{-\lambda_1t}\Bigl(\sum_{n=1}^\infty e^{-2(\lambda_n-\lambda_1)t}\Bigr)^{\f{1}{2}}\Bigl(\sum_{n=1}^\infty\f{|(a,\psi_n)|^2}{\rho_n^2}\Bigr)^{\f{1}{2}}\\
&+e^{-\mu_2t}\Bigl(\sum_{m=2}^\infty e^{-2(\mu_m-\mu_2)t}\Bigr)^{\f{1}{2}}\Bigl(\sum_{m=2}^\infty\f{|(a,\phi_m)|^2}{\sigma_m^2}\Bigr)^{\f{1}{2}}\\
&\leq e^{-\lambda_1t}\|a\|\Bigl(\sum_{n=1}^\infty e^{-2(\lambda_n-\lambda_1)t}\Bigr)^{\f{1}{2}}+e^{-\mu_2t}\|a\|\Bigl(\sum_{m=2}^\infty e^{-2(\mu_m-\mu_2)t}\Bigr)^{\f{1}{2}}.
\end{aligned}    
\end{equation*}
Utilizing asymptotic behaviour of eigenvalues
\begin{equation*}\label{lamays}
\sqrt{\lambda_n}=n\pi+\mathcal{O}\Bigl(\f{1}{n}\Bigr), \quad n\rightarrow \infty,\,\,\,
\sqrt{\mu_m}=m\pi+\mathcal{O}\Bigl(\f{1}{m}\Bigr),  \quad m\rightarrow \infty,
\end{equation*}
the estimate $$\sum_{n=1}^\infty e^{-2(\lambda_n-\lambda_1)t},\,\, \sum_{m=2}^\infty e^{-2(\mu_m-\mu_2)t}\leq C, \quad 0< t < \infty $$ hold with a constant $C>0$. Hence,
\begin{align}\label{lam1mu13}
\nn \f{|(a,\phi_{1})|}{\sigma_{1}}
\leq C\|a\|(e^{-(\mu_2-\mu_{1})t}
+e^{-(\lambda_1-\mu_{1})t})
\end{align}
for a constant $C>0$.
Therefore, if $\lambda_1>\mu_1$, then the right hand side tends to zero as $t\rightarrow\infty$. Hence we have $(a,\phi_{1})=0 $, which contradicts to the assumption.

Similarly, we can get
\begin{align*}
e^{-\lambda_{1}t}\f{|(a,\psi_{1})|}{\rho_{1}}\|\psi_{1}(0)\|_2
&\leq\sum_{n=2}^\infty e^{-\lambda_nt}\f{|(a,\psi_n)|}{\rho_n}\|\psi_n(0)\|_2+\sum_{m=1}^\infty e^{-\mu_mt}\f{|(a,\phi_m)|}{\sigma_m}\|\phi_m(0)\|_2\\
&\leq e^{-\lambda_2t}\Bigl(\sum_{n=2}^\infty e^{-2(\lambda_n-\lambda_2)t}\Bigr)^{\f{1}{2}}\Bigl(\sum_{n=2}^\infty\f{|(a,\psi_n)|^2}{\rho_n^2}\Bigr)^{\f{1}{2}}\nn\\
&+e^{-\mu_1t}\Bigl(\sum_{m=1}^\infty e^{-2(\mu_m-\mu_1)t}\Bigr)^{\f{1}{2}}\Bigl(\sum_{m=1}^\infty\f{|(a,\phi_m)|^2}{\sigma_m^2}\Bigr)^{\f{1}{2}}\nn\\
&\leq C\|a\|(e^{-\lambda_2t}
+e^{-\mu_1t}), \quad 0< t<\infty.
\end{align*}
Therefore, if $\mu_1>\lambda_1$, then $(a,\psi_{1})=0$, which is also a contradiction. Thus, we have proven $\lambda_1=\mu_1$. Substituting it into \eqref{data1expan} yields
$$e^{-\lambda_{1}t}\left(\f{(a,\psi_{1})}{\rho_{1}}\psi_{1}(0)-\f{(a,\phi_{1})}{\sigma_{1}}\phi_{1}(0)\right)=\sum_{m=2}^\infty e^{-\mu_mt}\f{(a,\phi_m)}{\sigma_m}\phi_m(0)-\sum_{n=2}^{\infty} e^{-\lambda_nt}\f{(a,\psi_n)}{\rho_n}\psi_n(0).$$
For $t>0$, the right-hand side of the above equality is estimated by $C(e^{-\lambda_2t}+e^{-\mu_2t})$ with a constant $C$.
Therefore, 
$$\left|\f{(a,\psi_{1})}{\rho_{1}}\psi_{1}(0)-\f{(a,\phi_{1})}{\sigma_{1}}\phi_{1}(0)\right|\le C\Bigl(e^{(\la_1-\la_2)t}+e^{(\la_1-\mu_2)t}\Bigr).$$
Since $\la_1<\la_2$ and $\la_1=\mu_1<\mu_2$, let $t\to \infty$, then
$$\f{(a,\psi_{1})}{\rho_{1}}\psi_{1}(0)=\f{(a,\phi_{1})}{\sigma_{1}}\phi_{1}(0)$$
follows and \eqref{data1expan} reads
\begin{equation*}
\sum_{n=2}^\infty e^{-\lambda_nt}\f{(a,\psi_n)}{\rho_n}\psi_n(0)=\sum_{m=2}^\infty e^{-\mu_mt}\f{(a,\phi_m)}{\sigma_m}\phi_m(0).
\end{equation*}
Continuing this procedure, we can derive
\begin{equation}\label{eigenequi}\lambda_n=\mu_n,\,\, \f{(a,\psi_{n})}{\rho_{n}}\psi_{n}(0)=\f{(a,\phi_{n})}{\sigma_{n}}\phi_{n}(0),\quad n\in \mathbb{N}.\end{equation}

By standard asymptotic analysis of the vector Sturm--Liouville problem with Neumann 
boundary conditions, the eigenfunctions admit the expansions
\begin{equation*}\label{psiays}
\psi_n(x) = \begin{pmatrix}
b^1_n \cos n\pi x + \mathcal{O}\bigl(\frac{1}{n}\bigr) \\[4pt]
b^2_n \cos n\pi x + \mathcal{O}\bigl(\frac{1}{n}\bigr)
\end{pmatrix}, \qquad
\phi_m(x) = \begin{pmatrix}
d^1_m \cos m\pi x + \mathcal{O}\bigl(\frac{1}{m}\bigr) \\[4pt]
d^2_m \cos m\pi x + \mathcal{O}\bigl(\frac{1}{m}\bigr)
\end{pmatrix}
\end{equation*}
uniformly for $x\in[0,1]$. In particular, evaluating at $x=1$ gives

\begin{equation*}
\|\psi_n(1)\|_2 = \left\| \begin{pmatrix}
b^1_n (-1)^n + \mathcal{O}\bigl(\frac{1}{n}\bigr) \\[4pt]
b^2_n (-1)^n + \mathcal{O}\bigl(\frac{1}{n}\bigr)
\end{pmatrix} \right\|_2 \le C, \quad
\|\phi_m(1)\|_2 = \left\| \begin{pmatrix}
d^1_m (-1)^m + \mathcal{O}\bigl(\frac{1}{m}\bigr) \\[4pt]
d^2_m (-1)^m + \mathcal{O}\bigl(\frac{1}{m}\bigr)
\end{pmatrix} \right\|_2 \le C
\end{equation*}
for a constant $C>0$ independent of $m$ and $n$. Utilizing the same procedure as above to \eqref{data2expan} implies
\begin{equation} \label{data22}
\f{(a,\psi_{n})}{\rho_{n}}\psi_{n}(1)=\f{(a,\phi_{n})}{\sigma_{n}}\phi_{n}(1),\quad n\in \mathbb{N}.
\end{equation}
From \eqref{eigenequi} and \eqref{data22}, there exist $c_n=\f{(a,\psi_{n})}{\rho_{n}}\cdot\f{\sigma_{n}}{(a,\phi_{n})}\neq 0$ such that
\begin{equation}\label{cn1}
\phi_n(0)=c_n\psi_n(0), \,\,
\phi_n(1)=c_n\psi_n(1),\quad n\in \mathbb{N}.     
\end{equation}

We denote by $\Phi(x;\lambda,\xi)$ the unique solution of \eqref{qslc} with initial value $\Phi(0)=\xi$; by linearity, 
$\Phi(x;\lambda,c\xi)=c\,\Phi(x;\lambda,\xi)$ for any constant $c$. 
Since the eigenfunction $\phi_n(x)$ satisfies the same differential equation 
with $\lambda=\lambda_n$ and the initial condition $\Phi(0)=\phi_n(0)$, $\Phi^{\prime}(0)=(0,0)^{\rm T}$, uniqueness 
implies $\Phi(x;\lambda_n,\phi_n(0))=\phi_n(x)$. Combining this with the relation 
\eqref{eigenequi} and \eqref{cn1}, we obtain 
for $\lambda=\lambda_n$,
\[
\Phi\bigl(x;\lambda_n, \psi_n(0)\bigr) 
= \Phi\Bigl(x;\lambda_n, \frac{\phi_n(0)}{c_n}\Bigr) 
= \frac{1}{c_n}\,\Phi\bigl(x;\lambda_n, \phi_n(0)\bigr) 
= \frac{1}{c_n}\,\phi_n(x).
\]
Substituting this into the transformation formula \eqref{tranfor} yields
\begin{equation}\label{tranforPcn}
\frac{1}{c_n}\,\phi_n(x) = \psi_n(x) + \int_0^x K(x,y)\,\psi_n(y)\,dy.
\end{equation}
Evaluating \eqref{tranforPcn} at $x = 1$ and using $\phi_n'(1) = 0$ together with 
\eqref{cn1}, we obtain
\begin{equation}\label{k1y}
\int_0^1 K(1,y)\,\psi_n(y)\,dy =  \begin{pmatrix} 0 \\ 0 \end{pmatrix}, \,\,
K(1,1)\,\psi_n(1) + \int_0^1 K_x(1,y)\,\psi_n(y)\,dy = \begin{pmatrix} 0 \\ 0 \end{pmatrix}, \,\,\,\, n \in \mathbb{N}.
\end{equation}

The family $\{\psi_n\}_{n=1}^\infty$ is complete in $\{L^2(0,1)\}^2$.  From the 
first equation in \eqref{k1y} and completeness we immediately deduce
\begin{equation}\label{k1y0}
K(1,y) = 0 \quad \text{for } y \in [0,1].
\end{equation}
In particular, $K(1,1) = 0$.  Substituting this into the second equation of 
\eqref{k1y} gives
\begin{equation}\label{Kx1y}
\int_0^1 K_x(1,y)\,\psi_n(y)\,dy = 0, \quad n \in \mathbb{N}.
\end{equation}
Applying completeness to \eqref{Kx1y} again yields
\begin{equation}\label{kx1y0}
K_x(1,y) = 0 \quad \text{for } y \in [0,1].
\end{equation}

Now consider the restriction of $K$ to the subdomain $\bar{\Omega}_2$. 
In view of \eqref{k1y0} and \eqref{kx1y0}, Proposition \ref{goursat2} implies 
$K \equiv 0$ on $\bar{\Omega}_2$.  In particular,
\begin{equation}\label{kx1-x}
K(x,1-x) = 0, \quad x \in [1/2,1].
\end{equation}
Furthermore, the kernel satisfies $K_y(x,0) = 0$ for $x\in [0,1]$ by 
Lemma \ref{lemma1}. Together with \eqref{kx1-x}, Proposition \ref{goursat4} 
yields $K \equiv 0$ on $\bar{\Omega}_1$.  Hence $K \equiv 0$ on the whole 
domain $\bar{D}$.

Finally, the relation $K(x,x) = \frac{1}{2}\int_0^x \bigl(Q(s) - P(s)\bigr)\,ds$ from Lemma \ref{lemma1} forces
\[
\frac{1}{2}\int_0^x \bigl(Q(s) - P(s)\bigr)\,ds = 0, \quad  x\in [0, 1],
\]
which implies $P(x) = Q(x)$ for $x \in [0,1]$.  This completes the proof 
of Theorem \ref{th1}.
\end{proof}

\section{Concluding remarks}\label{sec4}

In this paper, we considered the weakly coupled parabolic system 
$\partial_t u - \partial_x^2 u + P(x)u = 0$ subject to homogeneous Neumann 
boundary conditions, and proved the following uniqueness theorem: if the 
initial value $a(x)$ is a generating element with respect to both operators 
$\mathcal{A}_P$ and $\mathcal{A}_Q$, then the boundary observations 
$u(0,t)$ and $u(1,t)$ on any finite time interval $(0,T)$ uniquely determine 
the symmetric matrix coefficient $P(x)$. The paper also provides a detailed 
analysis of hyperbolic equations in the vector Sturm--Liouville setting, establishing the requisite existence, uniqueness and regularity estimates. Together, these results supply a rigorous theoretical foundation for the uniqueness of the inverse coefficient problem for weakly coupled parabolic systems, and offer a mathematical justification for the experimental determination of space-dependent coefficients in reaction-diffusion models. 

\section*{Acknowledgments}

The first author is supported by National Natural Science Foundation of China (no. 11601075), Fundamental Research Funds for the Central Universities of China (no. 2232015D3-35), and National Natural Science Foundation of China (no. 11526051).

\section*{Appendix}

In proving these propositions, we place the vertices at $A=(1,1)$, $B=(1,-1)$ and $C=(0,0)$ without loss of generality. Introducing the characteristic coordinates
\begin{equation}\label{cha var}
X = \frac{1}{2}(x+y), \qquad Y = \frac{1}{2}(x-y),  \tag{A.1}  
\end{equation}
the domain $\tri ABC$ is mapped onto the triangular region 
$\bar{\omega} = \{(X,Y)\mid X,Y\ge 0,\, X+Y\le 1\}$. 
Setting $k(X,Y) = K_v(X+Y, X-Y)$ and $R(X,Y) = \,r(X+Y, X-Y)$, 
then \eqref{K} becomes
\begin{equation}\label{k}
\partial_{XY}k(X,Y) = R(X,Y)\,k(X,Y), \quad (X,Y)\in \bar{\omega}. \tag{A.2}
\end{equation}

\subsection*{Proof of Proposition \ref{goursat1}.}

Using the change of variables \eqref{cha var}, the boundary conditions for \eqref{k} take the form
\begin{equation}\label{41bc}
k(X,0) = F(X)=f(X), \,\,X\in [0,1], \,\,\,k(0,Y) = G(Y)=g(Y), \,\, Y\in [0,1]. \tag{A.3}
\end{equation}
Equation $\partial_{XY} k=0$ in $\bar\om$ with the boundary conditions \eqref{41bc} possesses a unique 
solution $k_0\in C^2(\bar{\omega})$ given explicitly by d'Alembert's formula:
\begin{equation}\label{k410d}
k_0(X,Y) = F(X) + G(Y) - F(0). \tag{A.4}
\end{equation}
Moreover, for any $\Gamma\in C^0(\bar{\omega})$, the solution $k\in C^2(\bar{\omega})$ of the equation
\begin{equation*}
\begin{cases}
\partial_{XY} k=\Gamma, & (X,Y)\in \bar{\omega},\\
k(X,0)=0, & X\in [0,1],\\
k(0,Y)=0, & Y\in [0,1]
\end{cases}
\end{equation*}
is unique, and by Duhamel's principle, can be written as
\begin{equation*}\label{k41}
k(X,Y)=\int_0^X\int_0^Y\Gamma(\xi,\eta)d\eta d\xi=\iint_{\omega_1(X,Y)}\Gamma(\xi,\eta)d\xi d\eta,
\end{equation*}
where $\omega_1(X,Y)=\{(\xi,\eta)\mid 0<\xi <X, 0<\eta <Y\}.$ Consequently, the original problem \eqref{k} with boundary conditions \eqref{41bc} is equivalent to the Volterra integral equation
\begin{equation}\label{k41int}
k(X,Y)= k_0(X,Y)+\iint_{\omega_1(X,Y)}R(\xi,\eta)k(\xi,\eta)d\xi d\eta. \tag{A.5}
\end{equation}

Define the operator $\mathcal{T}: C^0(\bar{\omega})\to C^0(\bar{\omega})$ by 
$$(\mathcal{T}k)(X,Y)= k_0(X,Y)+\iint_{\omega_1(X,Y)}R(\xi,\eta)k(\xi,\eta)d\xi d\eta.$$
Since $R\in C^1(\bar{\omega})$, there exists a constant $M>0$ such that 
$\|R\|_{C^1(\bar{\omega})}\leq M$. For any $k,\tilde{k}\in C^0(\bar{\omega})$, 
we have the estimate
\begin{align*}
\|\mathcal{T}k- \mathcal{T}\tilde{k}\|_{C^0(\bar{\omega})}\leq MXY\|k-\tilde{k}\|_{C^0(\bar{\omega})}\leq M \|k-\tilde{k}\|_{C^0(\bar{\omega})}
\end{align*} 
in view of $X,Y\in[0,1]$. Furthermore, by induction one easily obtains
\begin{equation*}\label{induesti}
\|\mathcal{T}^nk- \mathcal{T}^n\tilde{k}\|_{C^0(\bar{\omega})}\leq M^n\f{X^nY^n}{n!n!}\|k-\tilde{k}\|_{C^0(\bar{\omega})}\leq \f{M^n}{n!n!}\|k-\tilde{k}\|_{C^0(\bar{\omega})}
\end{equation*}
Choosing $n$ sufficiently large so that $ M^n/(n!\,n!) < 1$, we see 
that $\mathcal{T}^n$ is a contraction on $C^0(\bar{\omega})$. By the 
Banach fixed-point theorem, \eqref{k41int} admits a unique solution 
$k\in C^0(\bar{\omega})$.

The solution $k$ to \eqref{k41int} can be constructed explicitly as the uniform limit of the 
successive approximations $\{k^{(n)}\}_{n\geq0}$ defined by
\begin{equation*}\label{itek0}
k^{(0)}(X,Y) \equiv 0, \qquad 
k^{(n+1)}(X,Y) = k_0(X,Y) + \int_0^X\!\!\int_0^Y R(\xi,\eta)\,k^{(n)}(\xi,\eta)\,d\eta\,d\xi.
\end{equation*}
Clearly each $k^{(n)}\in C^2(\bar{\omega})$. For brevity, we omit the arguments 
$(X,Y)$ of all functions, with the understanding that derivatives and integrals 
are evaluated pointwise. Differentiating the recurrence relation then yields
\begin{equation*}
\begin{aligned}
\partial_X k^{(n+1)} &= \partial_X k_0 + \int_0^Y R(\cdot,\eta)\,k^{(n)}(\cdot,\eta)\,d\eta, \\[2pt]
\partial_Y k^{(n+1)} &= \partial_Y k_0 + \int_0^X R(\xi,\cdot)\,k^{(n)}(\xi,\cdot)\,d\xi, \\[2pt]
\partial_X^2 k^{(n+1)} &= \partial_X^2 k_0 
+ \int_0^Y \bigl( R\,\partial_X k^{(n)} + (\partial_X R)\,k^{(n)} \bigr)(\cdot,\eta)\,d\eta, \\[2pt]
\partial_Y^2 k^{(n+1)} &= \partial_Y^2 k_0 
+ \int_0^X \bigl( R\,\partial_Y k^{(n)} + (\partial_Y R)\,k^{(n)} \bigr)(\xi,\cdot)\,d\xi, \\[2pt]
\partial_X\partial_Y k^{(n+1)} &= \partial_X\partial_Y k_0 + R\,k^{(n)}.
\end{aligned}    
\end{equation*}
Using the bound $\|R\|_{C^1(\bar{\omega})} \le M$, a straightforward induction 
argument yields the following estimates for the differences of consecutive iterates:
\begin{align}
\|k^{(n+1)} - k^{(n)}\|_{C^0(\bar{\omega})} 
&\le \frac{M^n}{n!\,n!}\,\|k_0\|_{C^0(\bar{\omega})},\nn \\[2pt]
\|\partial_X k^{(n+1)} - \partial_X k^{(n)}\|_{C^0(\bar{\omega})} 
&\le \frac{M^n}{n!}\,\|k_0\|_{C^1(\bar{\omega})}, \nn \\[2pt]
\|\partial_Y k^{(n+1)} - \partial_Y k^{(n)}\|_{C^0(\bar{\omega})} 
&\le \frac{M^n}{n!}\,\|k_0\|_{C^1(\bar{\omega})}, \nn\\[2pt]
\|\partial_X^2 k^{(n+1)} - \partial_X^2 k^{(n)}\|_{C^0(\bar{\omega})} 
&\le \frac{2M^n}{n!}\,\|k_0\|_{C^2(\bar{\omega})}, \nn\\[2pt]
\|\partial_Y^2 k^{(n+1)} - \partial_Y^2 k^{(n)}\|_{C^0(\bar{\omega})} 
&\le \frac{2M^n}{n!}\,\|k_0\|_{C^2(\bar{\omega})}, \nn \\[2pt]
\|\partial_X \partial_Y k^{(n+1)} - \partial_X \partial_Y k^{(n)}\|_{C^0(\bar{\omega})} 
&\le \frac{2M^n}{n!}\,\|k_0\|_{C^2(\bar{\omega})}.\nn
\end{align}
Here we use the fact that $X,Y\in [0,1]$ to replace the powers $X^nY^n$, $X^n$ and $Y^n$ by $1$. Based on the above results, we deduce that each of the telescoping series 
\[
\begin{aligned}
&\sum_{n=0}^\infty \bigl(k^{(n+1)} - k^{(n)}\bigr), \quad
\sum_{n=0}^\infty \bigl(\partial_X k^{(n+1)} - \partial_X k^{(n)}\bigr), \quad
\sum_{n=0}^\infty \bigl(\partial_Y k^{(n+1)} - \partial_Y k^{(n)}\bigr), \\[2pt]
&\sum_{n=0}^\infty \bigl(\partial_X^2 k^{(n+1)} - \partial_X^2 k^{(n)}\bigr), \quad
\sum_{n=0}^\infty \bigl(\partial_Y^2 k^{(n+1)} - \partial_Y^2 k^{(n)}\bigr), \quad
\sum_{n=0}^\infty \bigl(\partial_X \partial_Y k^{(n+1)} - \partial_X \partial_Y k^{(n)}\bigr)
\end{aligned}
\]
converges absolutely and uniformly on $\bar{\omega}$. Hence 
$k = \lim_{n\to\infty} k^{(n)}$ belongs to $C^2(\bar{\omega})$, and its 
derivatives are obtained by termwise differentiation. Moreover, $k$ satisfies 
\eqref{k} together with the boundary conditions \eqref{41bc}.
Furthermore, we get the estimates

\begin{equation} \label{estikc0}
\|k\|_{C^0(\bar{\omega})} \le e^M\,\|k_0\|_{C^0(\bar{\omega})},\, 
\|\partial_X k\|_{C^0(\bar{\omega})},\,
\|\partial_Y k\|_{C^0(\bar{\omega})}
\le e^M\,\|k_0\|_{C^1(\bar{\omega})},   \tag{A.6}      
\end{equation}
\begin{equation}\label{estikc2}
\|\partial_X^2 k\|_{C^0(\bar{\omega})},\;
\|\partial_Y^2 k\|_{C^0(\bar{\omega})},\;
\|\partial_X \partial_Y k\|_{C^0(\bar{\omega})}
\le 2e^M\,\|k_0\|_{C^2(\bar{\omega})}.  \tag{A.7}  
\end{equation}

Summing the above geometric series that arise from \eqref{estikc0}-\eqref{estikc2} 
yields the estimates
\begin{equation}\label{estc2}
\|k\|_{C^2(\bar{\omega})}\leq 2e^M \|k_0\|_{C^2(\bar{\omega})}. \tag{A.8}  
\end{equation}
Therefore, the estimate \eqref{estKc1} follows from \eqref{k410d} and \eqref{estc2} by the variable transformation immediately. This completes the proof of Proposition \ref{goursat1}.

\subsection*{Proof of Proposition \ref{goursat2}.}

The proof follows the same line of reasoning as that of Proposition~\ref{goursat1}. 
Under the same change of variables \eqref{cha var}, the boundary conditions \eqref{K42bc} is transformed into
\begin{equation}\label{42bc}
k(X,1-X) = F(X) = f(2X-1), \quad  X \in [0,1],\tag{A.9}  
\end{equation}
\begin{equation}\label{43bc}
\frac{1}{\sqrt{2}}\bigl( \partial_{X}k(1-Y,Y) + \partial_{Y} k(1-Y,Y) \bigr) = G(Y) = \sqrt{2}\,g(1-2Y), \quad Y \in [0,1]. \tag{A.10}  
\end{equation}
Equation \eqref{k} with the boundary conditions  \eqref{42bc}-\eqref{43bc} is equivalent to the 
Volterra-type integral equation
\begin{equation*}\label{k42int}
k(X,Y) = k_0(X,Y) + \iint_{\omega_2(X,Y)} R(\xi,\eta)\,k(\xi,\eta)\,d\xi d\eta,
\end{equation*}
where
\begin{equation*}\label{k420d}
k_0(X,Y) = F(X) - F(1-Y) + G(-Y)
\end{equation*}
and $\omega_2(X,Y) = \{(\xi,\eta) \mid  \xi < X,\; \eta > Y,\; \xi+\eta < 1\}$. 
From this point onward, the argument proceeds exactly as in the proof of Proposition \ref{goursat1}.

\subsection*{Proof of Proposition \ref{goursat4}.}

Similarly, under the same change of variables \eqref{cha var}, the boundary 
condition \eqref{K44bc} is transformed into

\begin{equation}\label{44bc}
k(X,0) = F(X) = f(X), \quad  X \in [0,1], \tag{A.11}  
\end{equation}
\begin{equation}\label{45bc}
\frac{1}{\sqrt{2}}\bigl( \partial_{X}k(1-Y,Y) + \partial_{Y} k(1-Y,Y) \bigr) = G(Y) = \sqrt{2}\,g(1-2Y), \quad  Y \in [0,1]. \tag{A.12}  
\end{equation}
Equation \eqref{k} with the boundary conditions \eqref{44bc}-\eqref{45bc} can be 
reformulated as the integral equation
\begin{equation*}\label{k44int}
k(X,Y) = k_0(X,Y) - \iint_{\omega_3(X,Y)} R(\xi,\eta)\,k(\xi,\eta)\,d\xi d\eta
- 2\iint_{\omega_4(X,Y)} R(\xi,\eta)\,k(\xi,\eta)\,d\xi d\eta,
\end{equation*}
where
\begin{equation*}\label{k440d}
k_0(X,Y) = F(X) + F(1-Y) - F(1) + \sqrt{2}\int_0^Y G(\eta)\,d\eta
\end{equation*}
and the regions are defined by 
$\omega_3(X,Y) = \{(\xi,\eta)\mid X < \xi < 1-Y,\, 0 < \eta < Y\}$ and 
$\omega_4(X,Y) = \omega_2(1-Y,0)$. 
This equation is again of Volterra type, and the existence, uniqueness, and regularity of the 
solution follow by the same fixed-point argument employed in Proposition \ref{goursat1}.

\bibliographystyle{unsrt}

\end{document}